\documentclass[a4paper,11pt]{article}

\usepackage[english]{babel}
\usepackage[T1]{fontenc}
\usepackage[utf8]{inputenc}
\usepackage[margin=2cm]{geometry}
\usepackage{amsmath,amssymb,amsthm}
\usepackage{graphicx}
\usepackage{paralist}

\def\NN{\mathbb{N}}
\def\ZZ{\mathbb{Z}}
\def\RR{\mathbb{R}}
\def\CC{\mathbb{C}}
\def\PP{\mathbb{P}}

\def\cA{\mathcal{A}}

\def\SL{\operatorname{SL}}

\def\sort{\operatorname{sort}}

\title{Some monoids of Pisot matrices}
\author{Artur Avila and Vincent Delecroix}
\newtheorem{theorem}{Theorem}
\newtheorem{lemma}[theorem]{Lemma}

\newtheorem{definition}[theorem]{Definition}

\begin{document}
\maketitle
\abstract{A matrix norm gives an upper bound on the spectral radius of a
matrix. Knowledge on the location of the dominant eigenvector also
leads to upper bound of the second eigenvalue.
We show how this technique can be used to prove that certain
semi-group of matrices arising from continued fractions have a Pisot
spectrum: namely for all matrices in this semi-group all eigenvalues
except the dominant one is smaller than one in absolute value.}

\section{Introduction}
A \emph{dominant eigenvalue} of a real square matrix is an eigenvalue of
maximum modulus. We call a square matrix \emph{Pisot} if it has
non-negative integer entries, its dominant eigenvalue is simple and all
eigenvalues different from the dominant one have absolute values less than one.
We prove that several monoids of non-negative matrices enjoy the
property of all being Pisot.

\medskip

Our first family of matrices is related to the so called \emph{fully subtractive} (multidimensional) continued fraction
algorithm. For an integer $d \geq 2$ we define
for each $k=1,\ldots,d$ the matrix $A_{FS,d}^{(k)}$ by
\[
(A_{FS,d}^{(k)})_{ij} = \left\{ \begin{array}{ll}
1 & \text{if $j=k$ or $i=j$}, \\
0 & \text{otherwise}
\end{array} \right.
\]
For $d=3$ this boils down to the three matrices
\[
A_{FS,3}^{(1)} = \begin{pmatrix}1&0&0\\1&1&0\\1&0&1\end{pmatrix},
\qquad
A_{FS,3}^{(2)} = \begin{pmatrix}1&1&0\\0&1&0\\0&1&1\end{pmatrix},
\qquad
A_{FS,3}^{(3)} = \begin{pmatrix}1&0&1\\0&1&1\\0&0&1\end{pmatrix}.
\]
All non-degenerate products of the matrices $A_{FS,d}^{(k)}$ satisfy the Pisot property.
\begin{theorem} \label{thm:fs}
Let $A = A_{FS,d}^{(i_1)} A_{FS,d}^{(i_2)} \ldots A_{FS,d}^{(i_n)}$ be a product of the fully subtractive matrices in dimension $d$.
Then the matrix $A$ is primitive if and only if all letters $\{1, \ldots, d\}$ appear in the sequence $(i_1, i_2, \ldots, i_n)$.
Moreover, if the matrix $A$ is primitive then it is Pisot.
\end{theorem}
Recall that a non-negative square matrix $A$ is primitive if there exists a positive integer $n$ so that $A^n$ has all its
entries positive. The case $d=3$ of Theorem~\ref{thm:fs} was proved in~\cite{ArnouxIto01}. The authors used an induction
on characteristic polynomials and our approach is radically different.

\medskip

The same result holds for another set of $3 \times 3$ matrices related to the Brun multidimensional
continued fractions. Let
\[
A_{Br}^{(1)} = \begin{pmatrix}1&1&0\\0&1&0\\0&0&1\end{pmatrix},
\qquad
A_{Br}^{(2)} = \begin{pmatrix}1&1&0\\1&0&0\\0&0&1\end{pmatrix},
\qquad
A_{Br}^{(3)} = \begin{pmatrix}1&0&1\\1&0&0\\0&1&0\end{pmatrix}.
\]
\begin{theorem} \label{thm:brun}
Let $A = A_{Br}^{(i_1)} A_{Br}^{(i_2)} \ldots A_{Br}^{(i_n)}$ be a product of the $3 \times 3$ Brun matrices.
Then, $B$ is primitive if and only if the matrix $B^{(3)}$ appears in the product.
Moreover, if $B$ is primitive then it is Pisot.
\end{theorem}
This result was already known since the work of Brun~\cite{Brun57}.

\medskip

The proofs of Theorem~\ref{thm:fs} and~\ref{thm:brun} are elementary and uses the following inequality.
Given a non-negative primitive $d \times d$ matrix $A$ and its Perron-Frobenius eigenvector $v \in \RR^d_+$ the absolute value $\lambda_2$
of its second largest eigenvalue satisfies
\[
\lambda_2 \leq \sup_{x \in v^\perp \backslash \{0\}} \frac{\|A x\|}{\|x\|}.
\]
where $\|.\|$ is any norm on $\RR^d$. In our proof, the information we have on the localization of the Perron-Frobenius eigenvector
comes from the dynamical systems induced by the matrices; in other words the fully subtractive and Brun continued fraction algorithms.

\medskip

From a diophantine approximation point of view,
the Pisot property is particularly interesting because it provides the so called
exponential convergence of the continued fraction
expansion for almost every vectors (see~\cite{Lagarias}). We show
that the above results naturally extends to this situation in Section~\ref{sec:lyapunov}.

\medskip

Beyond continued fractions, Pisot matrices are of special interest in
substitutive dynamical systems. More precisely, replacing matrices with so
called substitutions, the associated dynamical systems admit $d-1$ eigenvalues
and in many cases the dynamical system can be proved to have purely discrete
spectrum (see~\cite{Pytheas} Chapter~7). This was the main motivation for the
study of the fully subtractive matrices in~\cite{ArnouxIto01}.

\section{Fully subtractive and Brun continued fractions} \label{sec:cf}
Let us consider a finite or countable set $\cA$ that we call alphabet and for each $i \in \cA$
a matrix $A^{(i)} \in \SL(d,\ZZ)$. We already saw two examples of this with
the fully subtractive algorithm where $\cA_{FS,d} = \{1,2,\ldots,d\}$ and for Brun
algorithm where $\cA_{Br} = \{1,2,3\}$.

To the data $(\cA, (A^{(i)})_{i \in \cA})$ we associate the set of infinite words $\Delta = \cA^\NN$,
the shift map $T: \Delta \rightarrow \Delta$ and a cocycle
\[
\forall x \in \Delta, \forall n \geq 0, \quad A_n(x) = A^{(x_0)} A^{(x_1)} \ldots A^{(x_{n-1})}.
\]
The maps $A_n: \Delta \rightarrow \SL(d,\ZZ)$ satisfy the so called \emph{cocycle property}:
$A_{m+n}(x) = A_m(x) A_n(T^m x)$.

\begin{definition}
Let $(A^{(i)})_{i \in \cA}$ be a set of matrices in $\SL(d,\ZZ)$ where $\cA$ is a finite or countable alphabet.
We say that a set $D \subset \PP(\RR^d)$ is \emph{adapted} to these matrices if it is non-empty, it is the
closure of its interior and for all $i \in \cA$ we have $A^{(i)} D \subset D$.
\end{definition}
For example $PP(\RR^d)$ is always adapted. But we will be interested in the somewhat
smallest adapted set in order to localize the dominant eigenvector.

Given $(A^{(i)})_{i \in \cA}$ and $D \subset \PP(\RR^d)$ adapted, we define $D^{(i)} = A^{(i)} D$
and more generally for a finite word $w = i_0 i_1 \ldots i_{n-1}$ we define
$D^{(w)} = A^{(i_0)} A^{(i_1)} \ldots A^{(i_{n-1})} D$. Note that for any $w$ the set
$D^{(w)}$ is \emph{not} empty. Given an infinite word $x = x_0 x_1 \ldots \in \Delta$ we also
set $D_n(x) = A_n(x) D$ and $D_\infty(x) = \bigcap_{n \geq 0} D_n(x)$.

Let $A \in \SL(d,\ZZ)$ and let $\lambda$ be its spectral radius. Consider its
Jordan decomposition over $\CC$ and the Jordan blocks
associated with an eigenvalue of modulus $\lambda$ and being of maximal
dimension. To each of these maximal Jordan block is associated exactly one
eigenvector $v_i$. The \emph{dominant eigenspace} of $A$ is $\RR^d \cap (\CC
v_1 \oplus \CC v_2 \oplus \ldots \CC v_r)$. We have the following elementary
result.
\begin{lemma} \label{lem:dominant_eigenspace}
Let $(A^{(i)})_{i \in \cA}$ be a finite or countable set of matrices in $\SL(d,\ZZ)$ and
let $D$ be adapted. Let $x = x_0 x_1 \ldots \in \Delta$ be an infinite word over $\cA$.
Then for any $n$, the set $D_n(x)$ contains a basis of the dominant
eigenspace of $A_n(x)$.
\end{lemma}
We omit the proof that only uses the fact that the maximum growth of $\|A^n u\|$ is
$\lambda^n n^k$ where $k$ is the  maximal dimension of a Jordan block
associated with an eigenvalue of maximal modulus of $A$.

\medskip

Let
\[
D_{FS,d} = \{(x_1,\ldots,x_d) \in \PP(\RR_+^d): \forall i,j,k \quad x_i < x_j + x_k \}
\]
and
\[
D_{Br} = \{(x,y,z) \in \PP(\RR_+^3):\ x > y > z \}.
\]
Then it is easily seen that $D_{FS,d}$ is adapted for the fully subtractive matrices in dimension $d$
and $D_{Br}$ is adapted for the Brun matrices. In figures~\ref{fig:fs} and~\ref{fig:brun} one can see the projective picture
of the domains $D^{(1)}$, $D^{(2)}$ and $D^{(3)}$. Note that in these cases, the domains $D^{(i)}$ are
disjoint but that it is not a requirement in our definition. Moreover, one can see that in the Brun
case the $D^{(i)}$ form a partition while it is not the case for the fully subtractive.
\begin{figure}[!ht]
\centering
\begin{minipage}{0.4\textwidth}
\begin{gather*}
A_{FS,3}^{(1)}=\begin{pmatrix}1&0&0\\1&1&0\\1&0&1\end{pmatrix}
\qquad
A_{FS,3}^{(2)}=\begin{pmatrix}1&1&0\\0&1&0\\0&1&1\end{pmatrix}
\\
A_{FS,3}^{(3)}=\begin{pmatrix}1&0&1\\0&1&1\\0&0&1\end{pmatrix}.
\end{gather*}
\end{minipage}
\begin{minipage}{0.4\textwidth}
\centering \includegraphics{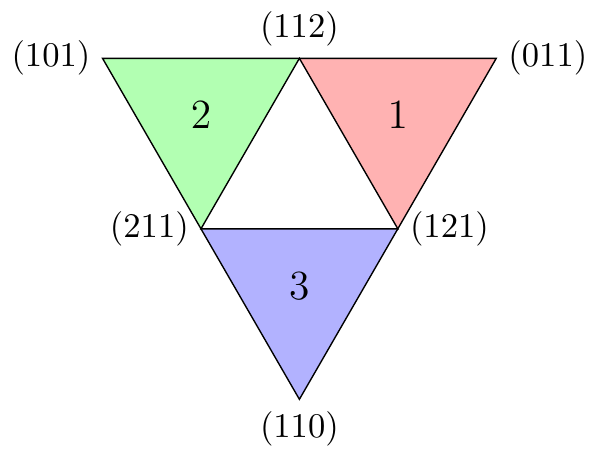}
\end{minipage}
\caption{Fully subtractive partition of the domains with $d=3$.}
\label{fig:fs}
\end{figure}

\begin{figure}[!ht]
\begin{center}
\begin{minipage}{0.4\textwidth}
\begin{gather*}
A_{Br}^{(1)} = \begin{pmatrix}1&1&0\\0&1&0\\0&0&1\end{pmatrix}
\qquad
A_{Br}^{(2)} = \begin{pmatrix}1&1&0\\1&0&0\\0&0&1\end{pmatrix}
\\
A_{Br}^{(3)} = \begin{pmatrix}1&0&1\\1&0&0\\0&1&0\end{pmatrix}.
\end{gather*}
\end{minipage}
\begin{minipage}{0.4\textwidth}
\begin{center}
\includegraphics{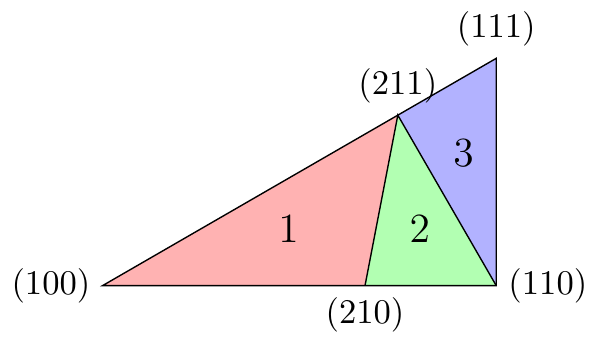}
\end{center}
\end{minipage}
\end{center}
\caption{The matrices and domains for the Brun algorithm.}
\label{fig:brun}
\end{figure}

If the $D^{(i)}$ are disjoint one can define a continued fraction algorithm
as follows. One defines a partial map $f: D \dashrightarrow D$ by setting
$fx = (A^{(i)})^{-1} x$ on $D^{(i)}$.

One can compute that for Brun one has
\[
f_{Br}(x) = \sort(x-y,y,z)
\]
where $\sort: \PP(\RR_+^3) \rightarrow D$ is the map which permutes the coordinates in order to sort them.
While for the fully subtractive one has
\[
f_{FS,d}(x) = (x_1-x_i, x_2-x_i, \ldots, x_{i-1}-x_i, x_i, x_{i+1}-x_i, \ldots, x_d - x_i)
\qquad
\text{if $x_i = \min(x_1, \ldots, x_d)$}.
\]

\section{Strategy}
The proofs of Theorems~\ref{thm:fs} and~\ref{thm:brun} follow a general strategy that we describe now.
We let $\|.\|$ be the $L^\infty$ norm on $\RR^d$ and the associated operator norm on matrices. That is
for a vector $v$ and a matrix $A$
\[
\|v\| = \max (|v_1|, \ldots, |v_d|)
\quad \text{and} \quad
\|A\| = \max_{i=1,\ldots,d} \sum_{j=1,\ldots,d} |a_{ij}|.
\]
In this section the norm used on $\RR^d$ has no importance. But it turns out
that, to apply the results to continued fraction algorithms, the most convenient one
was always the $L^\infty$ norm.

To a non-zero vector $v$ in $\RR^d$, we associate its dual hyperplane $H_v = \{z \in \RR^d; \langle v,z \rangle = 0\}$.
Given a non zero vector $v$ in $\RR^d$ we define the following semi-norm on $d \times d$ matrices
\[
\|B\|_v = \sup_{z \in H_v \backslash \{0\}} \frac{\|B z\|}{\|z\|} = \max_{\stackrel{\|z\| \leq 1}{z \in H_v}} \|B z\|.
\]
More generally, if $\Lambda \subset \RR^d$ is a cone, we define
\[
\|B\|_\Lambda = \sup_{v \in \PP(\Lambda)} \|B\|_v.
\]

Let $(A^{(i)})_{i \in \cA}$ be a finite or countable set of matrices as in Section~\ref{sec:cf}.
Let also $\Delta = \cA^\NN$ and $D \subset \PP(\RR^d)$ be adapted.
Recall that $D_n(x) = A_n(x) D = D^{(x_0 x_1 \ldots x_{n-1})}$ (in particular, $D_0(x) = D$ and $D_1(x) = D^{(x_0)}$)
and that $D_\infty(x) = \bigcap_{n \geq 0} D_n(x)$.
\begin{lemma}  \label{lem:strategy}
Let $(A^{(i)})_{i \in \cA}$ be a finite or countable set of matrices in $\SL(d,\ZZ)$ and let
$D \subset \PP(\RR^d)$ be adapted. Let $B^{(i)}$ (respectively
$B_n(x)$) denote the transposed of $A^{(i)}$ (resp. $A_n(x)$).
If for all $i \in \cA$ we have
\[
\left\| B^{(i)} \right\|_{D^{(i)}} \leq 1.
\]
Then for any point $x = x_0 x_1 \ldots \in \Delta$ we have
\[
\| B_n(x) \|_{D_n(x)} \leq 1.
\]
In particular, if $x = (x_0 x_1 \ldots x_{p-1})^\infty$ is periodic the matrix
$A_p(x) = A^{(x_0)} A^{(x_1)} \ldots A^{(x_{p-1})}$ has at most one eigenvalue
greater than one in absolute value.
\end{lemma}

\begin{proof}
Let $B$ be the transposed cocycle.
The hypothesis is just the case $n=1$. Assume that this inequality holds for $n$.
By definition $A_{n+1}(x) = A_1(x) A_n(Tx)$ and $D_{n+1}(x) = A_1(x) D_n(Tx)$.
Hence $v \in D_n(Tx)$ if and only if $A_1(x) v \in D_{n+1}(x)$.
Let us choose $v \in D_n(Tx)$, then
\[
\| B_{n+1}(x) \|_{A_1(x)v} = \| B_n(Tx) B_1(x) \|_{A_1(x)v} \leq \|B_n(Tx)\|_v \cdot \|B_1(x)\|_{A_1(x)v}.
\]

Now, let $x = (x_0 x_1 \ldots x_{p-1})^\infty$ be a periodic point. Because
$D_{p-1}(x) \supset D_\infty(x)$, the matrix $B_p(x)$ satisifies 
\[
\| B_p(x) \|_{D_\infty(x)} \leq 1.
\]
By Lemma~\ref{lem:dominant_eigenspace} a basis of the dominant eigenspace of $B_p$ belongs to
$D_\infty(x)$. Consequently, the union of the orthogonals of $D_\infty(x)$ contains
all eigenspaces corresponding to the non-dominant eigenvalues. From the above inequality,
we deduce that the absolute value of all eigenvalues different from the first one are bounded by $1$.
\end{proof}

\section{Pisot property for Arnoux-Rauzy matrices} \label{sec:fs}
In this section we prove Theorem~\ref{thm:fs}. Let $d \geq 2$ be an integer
and let $e_1$, $e_2$, \ldots $e_d$ be the canonical
basis of $\RR^d$. Let $e = e_1 + e_2 + \ldots + e_d$ and for $i=1,\ldots,k$ let $f_i = e - e_i$. 
The domain $D_{FS,d}$ is the convex hull of the rays vectors $\RR_+ f_i$.

Let as usual $B^{(w)} = (A^{(w)})^*$ and $B_n(x) = (A_n(x))^*$.
We claim that we even have a stronger property than what is required in Lemma~\ref{lem:strategy}
\[
\forall i=1,\ldots,d, \qquad \| B^{(i)} \|_D \leq 1.
\]
Let us prove this claim.
Let $v \in D_\infty$, then we may write $v = \mu_1 f_1 + \mu_2 f_2 + \ldots + \mu_d f_d$
for some non-negative numbers $\mu_i$ that satisfy $\mu_1 + \mu_2 + \ldots + \mu_d = 1$.
We hence have $v = e - \sum \mu_i e_i$ and
\[
H_v = \{z \in \RR^d; \langle z,v \rangle = 0\} = \{z \in \RR^d; \sum z_j = \sum \mu_j z_j\}.
\]
Given $z \in H_v$ we have $B^{(i)} z = (z_1, \ldots, z_{i-1}, \sum \mu_z z_j, z_{i+1}, \ldots, z_d)$,
In other words, $B^{(i)}$ acts on $H_v$ as a stochastic matrix $P(i,v)$ which is the identity except
its $i$-th row which is $(\mu_1,\mu_2,\ldots,\mu_d)$.
In particular, $\|B^{(i)}\|_v \leq 1$.

\medskip

Now for a given finite product $A = A^{(i_0)}A^{(i_1)} \ldots A^{(i_{p-1})}$ if one of the letter
$\{1,\ldots,d\}$ is missing in the sequence $(i_0, i_1, \ldots, i_{p-1})$ then $A e_i = e_i$ and
so the matrix is not primitive. On the other hand, if all letters appear it is easy to see
that all entries in $A$ are positive.

\medskip

Now let $x = (x_0 x_1 \ldots x_{p-1})^\infty$ be a periodic point that contains all letters from $\cA$.
Because of positivity, all the orbits is contained in the interior of $D$ and the dominant eigenvalue
is simple. Let $v_0 \in D$ be a dominant eigenvector and let $v_n = A_n(x) v_0$. Because all $(v_n)$
belongs to the interior of $D$ the coefficients $\mu_1$, $\mu_2$, \ldots, $\mu_d$ that appear in
the stochastic matrices $P(x_i, v_i)$ are all positive. Now, the product
$P(x_{p-1}, v_{p-1}) \ldots P(x_1, v_1) P(x_0, v_0)$ is a stochastic matrix with all its entries positive.
Hence, its second eigenvalue, which is also the second eigenvalue of $A_p(x)$, is less than $1$ in absolute
value.

\section{Pisot property for Brun algorithm (in dimension 3)}
We now turn to the proof of Theorem~\ref{thm:brun}.
Let $\cA = \{1,2,3\}$ and $A^{(1)}$, $A^{(2)}$, $A^{(3)}$ be the matrix of the Brun algorithm.
We let $\Delta = \cA^\NN$ and denote by $A$ and $B$ respectively the cocycle and the transposed
cocycle.
We claim that, as in the case of the fully subtractive, we have the stronger property that
\[
\forall i \in \{1,2,3\},\quad  \|B^{(i)}\|_{D^{(i)}} \leq 1.
\]
We only need to consider the matrix $\displaystyle B^{(1)} =
\begin{pmatrix}1&0&0\\1&1&0\\0&0&1\end{pmatrix}$ since the other two are
obtained by multiplying by a permutation matrix which will not change the $L^\infty$-norm.

Let $v = \mu_1 (1:0:0) + \mu_2 (1:1:0) + \mu_3 (1:1:1) \in D$ for some $\mu_1,\mu_2,\mu_3$ such that $\mu_1+\mu_2+\mu_3=1$
and $H_v = \{z \in \RR^d; z_1 + z_2 = \mu_1 z_2 - \mu_3 z_3 \}$.
Now, for any $z \in H_v$ we have $B^{(1)} (z_1,z_2,z_3) = (z_1,z_1+z_2,z_3) = (z_1, \mu_1 z_2 - \mu_3 z_3, z_3)$.
In other words $\|B^{(1)} (z_1,z_2,z_3)\|_1 \leq \|(z_1,z_2,z_3)\|$.

\medskip

Now, given a product $A = A_Br^{(i_1)} \ldots A_Br^{(i_n)}$ it is easy to see that if $3$ does not appear in
the sequence $(i_1, i_2, \ldots, i_n)$ then $A e_3 = e_3$ and hence the matrix $A$ can not be irreducible.
Conversly, if $3$ appears then $A^3$ is easily seen to be positive.

\medskip

Now consider the matrix $P$ built from the begining of the proof. As in the case of the fully subtractive
algorithm for a primitive product we got that the $\mu_i$ are all positive. Given a product $A$ where
the matrix $A_{Br}^{(3)}$ appears, the matrix $A^3$ is then such that all rows are such that sum of their
absolute values is strictly less than one. In other words $\|A^3\| < 1$.

\section{Lyapunov exponents} \label{sec:lyapunov}
Let $(A^{(i)})_{i \in \cA}$ be a finite or countable set of matrices.
Let $\Delta$, $T$, $A$, $B$ denote as before the infinite words, the shift map
the cocycle and the transposed cocycle. Let also $D$ be adapated to these matrices.

The asymptotic of the cocycle (or the transposed cocycle) are studied through \emph{Lyapunov
exponents}. Given a $T$-invariant ergodic probability measure $\mu$ on $\Delta$, we associate
the real numbers $\gamma^\mu_1 \geq \gamma^\mu_2 \geq \ldots \geq \gamma^\mu_d$ defined by
\[
\forall k \in \{1,2,\ldots,d\}, \quad
\gamma_1 + \gamma_2 + \ldots + \gamma_k = \lim_{n \to \infty} \int_\Delta \frac{\log \|\wedge^k A_n(x)\|}{n} d\mu(x).
\]
In order to be well defined we assume that
\begin{equation} \label{eq:integrability}
\int_\Delta \max \left(\log \|A_1(x)\|, \log \|A_1(x)^{-1}\|\right) d\mu(x) < \infty
\end{equation}
and we refer to this condition as the \emph{$\log$-integrability} of the cocycle.
If the alphabet $\cA$ is finite the cocycle is automatically $\log$-integrable.
If $x$ is a periodic point of $T$ and $\mu = (\delta_{x} + \delta_{Tx} + \ldots + \delta_{T^{n-1}x}) / n$ is
the sum of Dirac masses distribued along its orbit, then the associated Lyapunov exponents are the logarithms of 
the absolute values of eigenvalues of $A_n(x)$ where $n$ is the period of $x$.
In that sense, Lyapunov exponents generalize eigenvalues.

Given a measure $\mu$ for which the cocycle is $\log$-integrable, we say that $(\Delta,T,A,\mu)$ has \emph{Pisot spectrum}
if the associated Lyapunov exponents satisfy $\gamma_1^\mu > 0 > \gamma_2^\mu$. This property is related to the
strong convergence of higher dimensional continued fraction algorithm~\cite{Lagarias}.

Now we restate Lemma~\ref{lem:strategy} in a more dynamical context.
\begin{lemma} \label{lem:lyapunov}
Let $(A^{(i)})_{i \in \cA}$ be a finite or countable set of non-negative matrices in $\SL(d,\ZZ)$.
Let $(\Delta,T,A,B)$ be the associated full shift with its cocycle and its transposed cocycle.
Let also $D$ be adapated.
Assume that
\[
\forall i \in \cA, \quad  \left\| \left(A^{(i)}\right) \right\|_{D^{(i)}} \leq 1.
\]
Let $\mu$ be a $T$-invariant and ergodic measure on $D$ so that
\begin{itemize}
\item the cocycle $A_n$ is $\log$-integrable,
\item there exists a cylinder $[w]$ such that $\mu([w]) > 0$, $A^{(w)}$ is positive and $\|A^{(w)}\|_{D^{(w)}} < 1$.
\end{itemize}
Then two first Lyapunov exponents of the cocycle $A_n$ for the measure $\mu$ satisfies $\gamma_1^\mu > 0 > \gamma_2^\mu$.
\end{lemma}

\begin{proof}
Let us first prove that $\gamma_1 > 0$.

Now, by definition, for $\mu$-almost every $x$
\[
\gamma_1 = \lim_{n \to \infty} \frac{\log \|A_n(x)\|}{n}
\]
Let $m=|w|$ be the length of $w$ and consider the position which are multiple of $m$.
For a $\mu$-generic $x$ we have by Birkhoff theorem that
\[
\lim_{n \to \infty} \frac{ \# \{i \leq n:\ T^i(x) \in [w]\}} {n} = \mu([w])
\]
In other words, given sequence of length $n$ large enough we can find a linear number of
disjoint occurrences of $w$ (up to a sublinear error). Let $k_n$ be the number of these occurrences, then
necessarily each entry of $A_n(x)$ is larger than the corresponding one in
$C^{k_n}$ where $C$ is the matrix which contains a $1$ in every position.  In
particular $\gamma_1 > 0$.

From the existence of $w$ it also follows that for a $\mu$-generic $x$ the cone $D_\infty(x)$ is reduced to
a line contained in the interior of $\RR_+^d$. We can hence define $\mu$-almost everywhere a function
$v: \Delta \rightarrow \RR_+^d$ by $D_\infty(x) = \RR_+ v(x)$ and $\|v(x)\| = 1$.
We then have the following formulas which holds for $\mu$-almost every $x$
\[
\gamma_1 = \lim_{n \to \infty} \frac{- \log \|A_n(x)^{-1} v(x)\|}{n} d\mu(x)
\quad \text{and} \quad
\gamma_2 = \lim_{n \to \infty} \frac{\log \|B_n(x)\|_{v(x)}}{n} d\mu(x).
\]
It is then easy to derive the estimate for $\gamma_2$. The map $x \mapsto v(x)$ and the
dual hyperplanes $H_{v(x)}$ satisfy the following covariance properties
\[
D_\infty(Tx) = A(x)^{-1} D_\infty(x)
\quad \text{and} \quad
H_{A^{-1} v} = A^* H_v.
\]
Hence as in the proof of Lemma~\ref{lem:strategy}, we deduce that for all $v \in D_\infty(x)$
\[
\|B_{m+n}(x)\|_{v(x)} \leq \|B_{m}(x)\|_{v(x)} \|B_{n}(T^mx)\|_{v(T^m x)}
\]
In particular, for $\mu$-almost every $x \in [w]$, any $n \geq |w|$ we get that $\|B_n(x)\|_{v(x)} < 1$.
Let $\delta = \|B^{(w)}\|_{D^{(w)}} < 1$.
Using the same argument as in the estimation of $\gamma_1$ we get that
\[
\gamma_2 \leq \liminf_{n \to \infty} \frac{k_n \log \delta}{n}.
\]
And the above limit is strictly negative.
\end{proof}

\end{document}